\documentclass[11pt]{amsart}
\usepackage{amssymb,amscd,amsthm,verbatim}
\usepackage{amsbsy}
\usepackage{latexsym}

\usepackage[
colorlinks,
citecolor=red,
urlcolor=blue,
linkcolor=red,
allcolors=black,
backref=page
]{hyperref}

\date{\today}

\newcommand{\Z}{{\mathbb Z}}
\newcommand{\R}{{\mathbb R}}
\newcommand{\Q}{{\mathbb Q}}
\newcommand{\C}{{\mathbb C}}

\newcommand{\N}{{\mathbb N}}

\def\cC{\mathcal{C}}
\def\cA{\mathcal{A}}

\def\cW{\mathcal{W}}
\def\P{\mathbb{P}}

\newcommand{\CC}{{\mathcal{C}}}

\newcommand{\SL}{{\mathrm{SL}}}

\renewcommand{\sl}{{\mathfrak{sl}}}

\newtheorem{theorem}{Theorem}[section]
\newtheorem{lemma}[theorem]{Lemma}

\theoremstyle{definition}

\newtheorem*{definition}{Definition}

\newtheorem{remark}[theorem]{Remark}
\theoremstyle{plain}

\allowdisplaybreaks \numberwithin{equation}{section}

\begin{document}

\title{Pure Point Spectrum is Generic}

\author[A.\ Avila]{Artur Avila}

\address{Institut f\"ur Mathematik, Universit\"at Z\"urich, Winterthurerstrasse 190, 8057 Z\"urich, Switzerland and IMPA, Estrada D. Castorina 110, Jardim Bot\^anico, 22460-320 Rio de Janeiro, Brazil}

\email{artur.avila@math.uzh.ch}

\author[D.\ Damanik]{David Damanik}

\address{Department of Mathematics, Rice University, Houston, TX~77005, USA}

\email{damanik@rice.edu}

\thanks{D.\ D.\ was supported in part by NSF grants DMS--2054752 and DMS--2349919}


\begin{abstract}
We consider Schr\"odinger operators in $\ell^2(\Z)$ with real-valued potentials in $\ell^\infty(\Z)$ and show that the generic spectral type is pure point. More specifically, we show that for a generic bounded potential, the essential spectrum of the associated Schr\"odinger operator is a Cantor set and has zero weight with respect to all spectral measures.
\end{abstract}

\maketitle

\section{Introduction}

The title of this paper is intentionally provocative and is a play on a theme from the 1990's, namely that singular continuous spectrum is generic; see, for example, \cite{DJLS96, dRJMS94, dRMS94, HKS95, JS94, S95} and references therein.

For a bounded potential $V:\Z \to \R$, we consider the associated Schr\"odinger operator $H_V$ in $\ell^2(\Z)$, which acts as follows,
$$
[H_V \psi](n) = \psi(n+1) + \psi(n-1) + V(n) \psi(n),
$$
and is bounded and self-adjoint. We write $\ell^\infty(\Z)$ for the space of these potentials (instead of the more suggestive $\ell^\infty(\Z,\R)$), equipped with the uniform norm $\| \cdot \|_\infty$.

The genericity statement in the title of this paper refers to the following result:

\begin{theorem}\label{t.1}
For a generic $V \in \ell^\infty(\Z)$, the Schr\"odinger operator $H_V$ has pure point spectrum with exponentially decaying eigenvectors.
\end{theorem}

Here and throughout the paper, an assertion holds for a generic $V$ if there is a dense $G_\delta$ set of $V$'s for which the assertion is true.

\begin{remark}
Theorem~\ref{t.1} should be contrasted with the genericity result from the 1990's. It is a consequence of Simon's Wonderland Theorem from \cite{S95} that for every $R > 0$, a generic element $V$ of the product space $[-R,R]^\Z$ yields an operator $H_V$ that has purely singular continuous spectrum. It follows that there are two natural perspectives on Schr\"odinger operators with bounded potentials (namely with respect to pointwise convergence and uniform convergence, respectively) which result in drastically different generic spectral behavior.
\end{remark}

\begin{remark}
In other settings, the generic spectral type turns out to be singular continuous as well; compare, for example, \cite{AD05, AD24, BD08}. It is notable that Theorem~\ref{t.1} exhibits the first setting (where other spectral types are possible and) in which pure point spectrum is generic.
\end{remark}

Theorem~\ref{t.1} is an immediate consequence of the following result:

\begin{theorem}\label{t.2}
For a generic $V \in \ell^\infty(\Z)$, all spectral measures of the associated Schr\"odinger operator $H_V$ assign zero weight to $\sigma_\mathrm{ess}(H_V)$.
\end{theorem}

This theorem shows that for a generic $V \in \ell^\infty(\Z)$, the spectral measures are supported by the discrete spectrum, which consists of eigenvalues that have associated exponentially decaying eigenfunctions, and thus Theorem~\ref{t.1} follows. 

It also follows from Theorem~\ref{t.2} that for a generic $V \in \ell^\infty(\Z)$, the essential spectrum of $H_V$ must have empty interior. In addition, we will show that the absence of isolated points in the essential spectrum is a generic property, and hence we will establish the following result:

\begin{theorem}\label{t.3}
For a generic $V \in \ell^\infty(\Z)$, the essential spectrum of  the associated Schr\"odinger operator, $\sigma_\mathrm{ess}(H_V)$, is a Cantor set.
\end{theorem}

As usual, a subset of $\R$ is called a Cantor set if it is compact, has no isolated points, and has empty interior.

\begin{remark}
The absence of isolated points in the essential spectrum does not always hold. For an example, choose a sparse potential of the form $V(n) = \sum_{k = 1}^\infty \delta_{n, k!}$. Then, it follows from \cite{LS06} that $\sigma_\mathrm{ess}(H_V) = [-2,2] \cup E_+$, where $E_+ > 2$ is the unique discrete eigenvalue of $H_{\tilde V}$, where $\tilde V(n) = \delta_{n,0}$. This example also shows that the essential spectrum may contain intervals, which is of course obvious as already the free case ($V \equiv 0$) has this property.
\end{remark}

\begin{remark}
It would be interesting to investigate whether the Lebesgue measure of the essential spectrum associated with a generic potential is zero.  At the moment we do not have an informed guess to make. This problem is in the same spirit as the analogous question in the case of dynamically defined potentials. For example, it is known that in the quasi-periodic case (and indeed many other cases as well), a generic continuous sampling function leads to an operator family with a Cantor spectrum \cite{ABD09, ABD12}, but it is not known whether the Lebesgue measure of this set vanishes generically as well. Some related results on obtaining zero-measure spectrum via small $\|\cdot\|_\infty$ perturbations are contained in \cite{CDFG22, DL06}.
\end{remark}

%\section{Statements of the results}

%\begin{theorem} \label {discrete}
%For generic $V \in \ell^\infty(\Z)$, the spectrum of $H_V$ is pure point, and all eigenvalues are outside the essential spectrum.
%\end{theorem}

%\begin{theorem} \label {cantor}
%For generic $V \in \ell^\infty(\Z)$, the essential spectrum is Cantor.
%\end{theorem}

The remainder of the paper is structured as follows. In Section~\ref{sec.2} we reduce Theorem~\ref{t.2} to an effective perturbation result, Theorem~\ref{malleable}, which may be of independent interest. Theorem~\ref{malleable} is then proved in Section~\ref{sec.3}. Finally, we prove Theorem~\ref{t.3} in Section~\ref{sec.4}.

\section{Proof of Theorem \ref{t.2} Modulo Effective Perturbation}\label{sec.2}

In this section we reduce Theorem~\ref{t.2} to Theorem~\ref{malleable}, which is formulated below. 

\begin{definition}
We say that $V \in \ell^\infty(\Z)$ is $\epsilon$-\emph{discrete} if the spectral measure\footnote{When we refer to ``the'' spectral measure, we have in mind (one half) the canonical spectral measure in the sense of \cite{DF22}, that is, the arithmetic average of the spectral measures corresponding to $H$ and the vectors $\delta_0,\delta_1$. Since the pair $\{ \delta_0, \delta_1 \}$ is cyclic, the spectral measure of every $\psi \in \ell^2(\Z)$ is absolutely continuous with respect to the canonical spectral measure.} of the essential spectrum is less than $\epsilon$, and $V$ is said to be of \emph{discrete type} if it is $\epsilon$-\emph{discrete} for every $\epsilon>0$, that is, if  the spectral measure of the essential spectrum is zero.
\end{definition}

Theorem~\ref{t.2} can be reformulated by saying that a generic $V \in \ell^\infty(\Z)$ is of discrete
type. The set of $V$ that are $\epsilon$-discrete is open for every $\epsilon > 0$.\footnote{The essential spectrum of $H_V$ depends continuously on $V$ in the Hausdorff topology and the spectral measure depends continuously on $V$ in the weak-$*$ topology. Thus, the map $\ell^\infty(\Z) \ni V \mapsto \mu(\sigma_\mathrm{ess}(H_V))$ is upper semicontinuous.}  So in order to prove Theorem~\ref{t.2} it is enough to show that $\epsilon$-discrete potentials are dense.

Let us denote $A^{(t)}=\begin{pmatrix} t & -1\\1&0 \end{pmatrix}$.

\begin{definition}
(a) For $N,T \geq 1$, let $\Phi^{N,T}:\R^N \times \R^T \to \SL(2,\R)^N$ be given by $\Phi^{N,T}=(\Phi^{N,T}_j)_{j=1}^N$, $\Phi^{N,T}_j(E,V)=A^{(E_j-V_T)} \cdots A^{(E_j-V_1)}$.
\\[1mm]
(b) We let $\cC=\cC^N \subset \R^N$ be the set of all $E \in \R^N$ with $E_i=E_j$ for some $1 \leq i<j \leq N$.
\\[1mm]
(c) We say that $V^N \in \R^T$ is $N$-\emph{malleable} if for every $E \in \R^N \setminus \cC^N$, $V \mapsto \Phi^{N,T}(E,V)$ is open at $V^N$.
\end{definition}

%$w=(w_1,...,w_N) \in \SL(2,\R)^N$ is $\delta$-close to $\Phi^{N,T}(E,V)$
%then there exists $V' \in \R^T$ which is $\epsilon$-close to $V$ such that
%$\Phi^{N,T}(E,V')=w$.

The following simultaneous perturbation
result plays an important role in our construction.  Its proof
is rather technical and will be given in the next section.

\begin{theorem} \label {malleable}

For every $N \geq 1$, there exist $T \geq 1$ and a Baire generic subset of $V \in \R^T$ which are $N$-malleable.

\end{theorem}

We do not know whether there exists $T \geq 1$ such that all
$V \in \R^T$ are $N$-malleable.

\begin{definition}
We say that $V \in \ell^\infty(\Z)$ is $\infty$-\emph{malleable} if for every $N \geq 1$, there
exist some $T_N \geq 1$ and $V^N \in \R^{T_N}$ such that $V^N+t$ is $N$-malleable for every $t \in \Q^{T_N}$, and there exists a finite subset $Q_N \subset V^N+\Q^{T_N}$ such that the set $L_N$ of $l \in \Z$ with
$(V_{l+1},...,V_{l+T_N}) \in Q_N$ is unbounded from below and from above.
\end{definition}

%large in the following sense: for every
%$\epsilon>0$ there exists $C_\epsilon>0$ such that for every
%$n \in \Z$ there exists $m \in L_N$ such that $|m-n|<\epsilon
%|n|+C_\epsilon$.

\begin{remark}\label{r.malleable}
It follows from Theorem~\ref{malleable} that the $\infty$-malleable $V \in \ell^\infty(\Z)$ are dense. 
Indeed, suppose that $V \in \ell^\infty(\Z)$ and $\varepsilon > 0$ are given. 
%Since $V$ is bounded, for each $N \ge 1$, there are balls $B^N_\pm$ in $\R^N$ of radius, say, $1$  and sequences $\ell^N_{k,\pm} \to \pm \infty$ such that the sets $\{ \ell^N_{k,\pm}, \ldots, \ell^N_{k,\pm}+N-1\}$ are mutually disjoint (as $k, N \in \N$ and the $\pm$ sign are being varied) and $(V(\ell^N_{k,\pm}),\ldots,V(\ell^N_{k,\pm}+N-1)) \in B^N_\pm$. 
For each $N \ge 1$, choose $T_N \ge 1$ via Theorem~\ref{malleable} such that a generic $V^N \in \R^{T_N}$ is $N$-malleable. Since the translate of a generic set is generic, we preserve genericity if we subtract a fixed $t \in \Q^{T_N}$ from $V^N$. Intersecting these countably many generic sets, we obtain that every $V^N$ in a generic subset $S^N \subseteq \R^{T_N}$ is such that $V^N+t$ is $N$-malleable for every $t \in \Q^{T_N}$. Fix an arbitrary element $V^N$ of $S^N$. 
For each $N \ge 1$, choose sequences $\ell^N_{k,\pm} \to \pm \infty$ such that the sets $\{ \ell^N_{k,\pm}, \ldots, \ell^N_{k,\pm}+T_N-1\}$ are mutually disjoint (as $k, N \in \N$ and the $\pm$ sign are being varied).
The claim follows upon choosing a suitable finite subset $F^N$ of $V^N + \Q^{T_N}$ that is $\varepsilon$-dense in $\{ v = (v_1,\ldots,v_N) \in \R^{T_N} : \max_j |v_j| \le \|V\|_\infty \}$ and replacing the strings $(V(\ell^N_{k,\pm}),\ldots,V(\ell^N_{k,\pm}+T_N-1)))$ by suitable $\varepsilon$-close elements of $F^N$. Since the sets $\{ \ell^N_{k,\pm}, \ldots, \ell^N_{k,\pm}+T_N-1\}$ are mutually disjoint and the modifications on them are bounded in size uniformly by $\varepsilon$, it follows that the $V' \in \ell^\infty(\Z)$ resulting from these modifications is $\infty$-malleable and obeys $\|V' - V\|_\infty \le \varepsilon$.

\end{remark}

We will need the following result, which follows from the work \cite{GK24} of Gorodetski-Kleptsyn.\footnote{Theorem~\ref{localization} does not appear in this formulation in \cite{GK24}, but it follows from their work; see specifically Proposition~3.8, Definition~4.3, and the discussion above Lemma~4.6 in \cite{GK24}.}

\begin{theorem} \label {localization}

For every $V \in \ell^\infty(\Z)$, for every $r>0$, for almost every $t
\in \{0,r\}^\Z$ with respect to the Bernoulli measure,
$V+t$ has pure point spectrum and the eigenfunctions are exponentially
localized with semi-uniform decay.  More precisely, there exists
$c>0$ such that for any
eigenfunction $\phi$, for every $\delta>0$, there exists
$C_{\delta,\phi}$ such that $u_k=\begin{pmatrix} \phi_k \\
\phi_{k-1} \end{pmatrix}$ satisfy $\|u_{\pm (n+m)}\| \leq
C_{\delta,\phi} e^{\mp (c m-\delta n)} \|u_{\pm n}\|$ for $n,m \geq 0$.

\end{theorem}

Let us note the following (much easier) result can be proved directly, but it is also an obvious consequence of the work \cite{GK24} of Gorodetski-Kleptsyn:

\begin{lemma} \label {growing}

For $\lambda>0$, for every $N \geq 1$, for every $\epsilon>0$, there exist $\delta>0$, $C>0$, such that for every $m \geq 1$ and for every $E \in \R^N$, $V \in \R^m$ with $\|E\|_\infty, \|V\|_\infty < \lambda$, there exists $V' \in \R^m$ which is
$\epsilon$-close to $V$ such that $\|\Phi^{N,m}_j(E,V')\| \geq C e^{\delta m}$, $1 \leq j \leq N$.

\end{lemma}

For $V \in \ell^\infty(\Z)$, $m,n \in \Z$, let $S^{(E,V)}_{m,n}$ be the transfer matrix from $m$ to $m+n$, so $S^{(E,V)}_{m,1}=A^{(E-V_m)}$ and $S^{(E,V)}_{m+k,l} S^{(E,V)}_{m,k}=S^{(E,V)}_{m,k+l}$.  For an interval
$[m,m+n-1] \subset \Z$, $n \geq 1$, we will also use the notation $S^{(E,V)}_I=S^{(E,V)}_{m,n}$.

Notice that if $V$ is $\infty$-malleable and $Q \subset \Q$ is finite, then for every $t \in Q^\Z$, $V+t$ is $\infty$-malleable.

By Remark~\ref{r.malleable} and Theorem~\ref{localization}, there exists a dense subset of $\infty$-malleable potentials that have pure point spectrum with exponentially localized eigenfunctions with semi-uniform decay. Let $V$ be such a potential.  For $\epsilon>0$ we will now show how to $\epsilon$-perturb $V$ to obtain a potential of $\epsilon$-discrete type.

Fix finitely many eigenvalues $E_1<...<E_N$ such that the spectral measure (for $H_V$) of $\{E_1,...,E_N\}$ is larger than $1-\epsilon$.  Let $E=(E_1,...,E_N)$.  Let $\phi_j=(\phi_{j,k})_{k \in \Z} \in \ell^\infty(\Z)$ be
the corresponding $\ell^2$-normalized eigenfunctions, $1 \le j \le N$.  We denote $u_{j,k}=\begin{pmatrix} \phi_{j,k} \\ \phi_{j,k-1} \end{pmatrix}$.  Particularly we have
\begin{equation} \label {4}
\|u_{j,n}\| \leq C_0 e^{-c |n|}, \quad n \in \Z.
\end{equation}

Let $T=T_N$, $Q=Q_N$ and $L=L_N$ be associated to $V$ as in the definition of a $\infty$-malleable potential.

The construction will depend on a sequence of parameters chosen subsequently; first some small $\rho>0$, then some $m_1 \geq 1$ large, then some $m_2 \geq 1$ large.

We will decompose the non-negative integers $\Z^{\geq 0}$ into several blocks. There is a starting block $I^+_s$, a malleable block $I^+_m$, a decay block $I^+_d$, and two alternating infinite sequences of buffer blocks $I^+_{b,i}$ and hyperbolic blocks $I^+_{h,i}$, $i \geq 0$, defined as follows. Let $l_+ \in L$ be minimal with $l_+ \geq m_2$, $m_+=l_+ +[\rho l_+]$. Then set $I^+_s=[0,l_+]$, $I^+_m=[l_+ +1,l_+ +T]$,
$I^+_d=[l_+ + T + 1,m_+]$, $I^+_{b,i}=[m_+ +i m_1+1,m_+ +i m_1+2]$, $I^+_{h,i}=[m_+ +i m_1+3,m_+ +(i+1) m_1]$, $i \geq 0$.

We perturb $V$ to $V'$ on $\Z^{\geq 0}$ separately on each block as follows:

\begin{itemize}

\item We do not perturb on the starting and decay blocks.

\item On a hyperbolic block $I^+_{h,i}$, $i \geq 0$, we perturb so that $S^{(E_j,V')}_{I^+_{h,i}}$ is large, $1 \leq j \leq N$.  This is possible by Lemma \ref {growing}. The largeness depends on the allowed perturbation size $\epsilon$ and $m_1$, and improves by taking $m_1$ larger.

\item On a buffer block $I^+_{b,i}$, $i \geq 1$, we perturb so that the most contracted direction by
the inverse of $S^{(E_j,V')}_{I^+_{h,i-1}}$ is taken by $S^{(E_j,V')}_{I^+_{b,i}}$ to a
direction making angle at least $\rho$ with the most contracted direction by
$S^{(E_j,V')}_{I^+_{h,i}}$, $1 \leq j \leq N$. This is possible since the length of the buffer block is $2$ and (compare the discussion in \cite[Remark~1.7]{GK21}) $\inf \{ \frac{d}{dt} \mathrm{Arg} S^{(E,V-t)}_{m,2} v : m \in \Z, \, \|E\|_\infty, \|V\|_\infty < \lambda , \, v \in \R^2 \setminus \{0\} \}> 0$.

\end{itemize}

Since $m_1$ is chosen after $\rho$, the largeness of matrices associated to
hyperbolic blocks compensates for the angle estimate and allows hyperbolic
concatenation.  We conclude that there exist directions $s^+_j \in
\P\R^2$, $1 \leq j \leq N$, such that
\begin{equation} \label {1}
\|S^{(V',E_j)}_{m_+ +3,n} \cdot w\| \leq C_1 e^{-\delta_1 n} \|w\|,
\quad n \geq 1, \quad w \in s^+_j.
\end{equation}

\begin{itemize}

\item On the buffer block $I^+_{b,0}$ we perturb so that the most contracted
direction by the inverse of
$S^{(E_j,V')}_{I^+_d}$ is sent by
$S^{(E_j,V')}_{I^+_{b,0}}$ to a
direction making angle at least $\rho$ with $s^j_+$, $1 \leq j \leq N$.

\end{itemize}

For $1 \leq j \leq N$, let
$\tilde s^+_j$ be the image of $s^+_j$ by the inverse of
$S^{(E_j,V')}_{I^+_{b,0}} S^{(E_j,V')}_{I^+_d}$.

The semi-uniform property in Theorem \ref {localization} implies that a decay block produces exponential contraction of the eigenfunctions, in the sense that $|u_{j,m_+ +1}| \leq e^{-c \rho l_+/2} |u_{j,l_+ +T+1}|$ (provided $m_2$ is sufficiently large).  This implies that the direction of $u_{j,l_+ +T+1}$ must be exponentially (in $\rho m_2$) close to the most contracted direction of $S^{(E_j,V')}_{I^+_d}$, which is in turn exponentially close to $\tilde s^j_+$ (by the choice of the buffer block $I^+_{b,0}$).

\begin{itemize}

\item We conclude that if $m_2$ is sufficiently large, then we can perturb in the malleable block $I^+_m$ so that
$S^{(E,V')}_{I^+_m}$ takes $u_{j,l_+ +1}$ to a vector in $\tilde s^+_j$, $1 \leq j \leq N$.

\end{itemize}

Let $u'_{j,n}=S^{(E_j,V')}_{0,n} u_{j,0}$, $n \geq 0$. Since we did not change $V$ on $I^+_s$, we have
\begin{equation} \label {5}
u'_{j,n}=u_{j,n}, \quad 0 \leq n
\leq l_+ + 1.
\end{equation}
The changes in $I^+_m$ are so that $u'_{j,l_+ +T+1}$ belongs to $\tilde s^+_j$, so that $u'_{j,m_+ +3}$ belongs to $s^+_j$.

Using (\ref {4}) and (\ref {5}) together with
the trivial bound $\|u'_{j,n+1}\| \leq C_2 \|u'_{j,n}\|$ and the choice
of small $\rho$ we get
\begin{equation}
\|u'_{j,n}\| \leq e^{-c n/3}, \quad 
l_+ +1 \leq n \leq m_+ +3.
\end{equation}
Together with (\ref {1}) we get, with $c'=\min \{c/4,\delta_1/2\}$,
\begin{equation} \label {7}
\|u'_{j,n}\| \leq e^{-c' n}, \quad 
n \geq l_+ +1.
\end{equation}

We now proceed similarly to perturb $V$ on the negative integers $\Z^{<0}$. 
Let $l_- \in L$ be maximal with $l_-+T+1 \leq -m_2$,
$m_-=l_- -[\rho l_+]$.
Then set $I^-_s=[l_-,-1]$, $I^-_m=[l_- -T,l_- -1]$,
$I^-_d=[m_-,l_- -T-1]$, $I^-_{b,i}=[m_- -i m_1-2,m_- -i m_1-1]$,
$I^-_{h,i}=[m_- -(i+1) m_1,m_- -i m_1-3]$.

\begin{itemize}

\item We do not perturb on starting and decay blocks.

\item On the hyperbolic blocks we perturb to make $S^{(E_j,V')}_{I^-_{h,i}}$
large.

\item On the buffer blocks with $i \neq 0$ we perturb so that the most contracted direction by
the inverse of $S^{(E_j,V')}_{I^-_{h,i}}$ is taken by $S^{(E_j,V')}_{I^-_{b,i}}$ to a
direction making angle at least $\rho$ with the most contracted direction by
$S^{(E_j,V')}_{I^-_{h,i-1}}$.

\end{itemize}

Hyperbolic concatenation allows us then
to define $s^-_j \in \P \R^2$ satisfying
\begin{equation} \label {1-}
\|S^{(V',E_j)}_{m_- -2,-n} \cdot w\| \leq C_1 e^{-\delta_1 n} \|w\|,
\quad n \geq 1, \quad w \in s^-_j.
\end{equation}

\begin{itemize}

\item On the buffer block with $i=0$ we perturb so that the most contracted direction by $S^{(E_j,V')}_{I^-_d}$ is sent by the inverse of $S^{(E_j,V')}_{I^-_{b,0}}$ to a direction making angle at least $\rho$ with $s^-_j$.

\end{itemize}

We can then conclude that the image of $s^-_j$ by $S^{(E_j,V')}_{I^-_d} S^{(E_j,V')}_{I^-_{b,0}}$ is a direction $\tilde s^-_j$ exponentially close to the most contracted direction by the inverse of $S^{(E_j,V')}_{I^-_d}$,
which is exponentially close to the direction of $u_{j,l_- -T}$.  

\begin{itemize}

\item We then perturb on the malleable block so that $S^{(E_j,V')}_{I^-_b}$ takes $\tilde s^-_j$ to the direction of $u_{j,l_-}$.

\end{itemize}

The definition of $V'$ has now been determined on the whole of $\Z$. We can then define $u'_{j,n} = S^{(E_j,V')}_{0,n} u_{j,0}$ for $n<0$ as well.  We can then get as before
\begin{equation} \label {5-}
u'_{j,n}=u_{j,n}, \quad l_- \leq n<0,
\end{equation}
\begin{equation} \label {7-}
\|u'_{j,n}\| \leq e^{c' |n|}, \quad 
n \leq l_- -1.
\end{equation}

We can write $u'_{j,n}=\begin{pmatrix} \phi'_j \\ \phi'_{j-1} \end{pmatrix}$ where $H_{V'} \phi'_j=E_j \phi_j$, where $\phi'_j \in \ell^2(\Z)$ with $\|\phi_j-\phi'_j\|^2_{\ell^2(\Z)} \leq C' e^{-2 c' m_2}$.

This implies that the $V'$-spectral measure and the $V$-spectral measure of $E_j$ are close, $1 \leq j \leq N$.  To conclude that $V'$ is of $\epsilon$-discrete type it is thus enough to show that the $E_j$ are not
in the essential spectrum of $H_{V'}$. By general theory, this is equivalent to the existence of non-zero
vectors $w^\pm$ that decay exponentially uniformly under
$S^{(E_j,V')}_{0,\pm n}$.\footnote{One direction follows from the
Combes-Thomas estimate, which gives uniform exponential decay of the
Green function, and hence of the Weyl solutions, for energies outside
the spectrum; compare \cite[Theorem~2.5.1]{DF22}. Note that discrete
eigenvalues can be removed by a small local perturbation under which
the essential spectrum is invariant. In the other direction, one can
again break a possible coincidence of $w^\pm$ via a local perturbation
and then express the resolvent using the pair of
Weyl solutions.}  Here by uniform decay we mean that $\|S^{(E_j,V')}_{0,\pm (n+m)}
w^\pm\| \leq C' e^{-c' m} \|S^{(E_j,V')}_{0,\pm n} w^\pm\|$ for $n,m \geq
0$.  By (\ref {1}) and (\ref {1-}), this is precisely what we get with $w^+=w^-$, the direction of $u'_{j,0}$.
\qed

\medskip

Thus, Theorem~\ref{t.2} has been proved modulo the proof of Theorem~\ref{malleable}.

\section{Effective Perturbations}\label{sec.3}

In this section we prove Theorem~\ref{malleable}. Let
$\cA=\cA^{N,T} \subset \R^N \times \R^T$ be the set of all
$(E^*,V^*)$ such that $V \mapsto \Phi(E^*,V)$ is not a submersion at $V^*$.
For $E \in \R^N$, let $\cA_e(E)=\cA^{N,T}_e(E) \subset \R^T$ be the set of
all $V$ such that $(E,V) \in \cA$.  For $V \in \R^T$, let
$\cA_v(V)=\cA^{N,T}_v(V) \subset \R^N$ be the set of all $E$ such that
$(E,V) \in \cA$.

Clearly $\cA$, $\cA_e$ and $\cA_v$ are all real algebraic subsets.
Moreover, $\cC \subset \cA_v$.  We will repeatedly use the following
concatenation property.  Let $T_j$, $1 \leq j \leq k$, and let
$T = \sum_{j=1}^k T_j$.  Then for every $V^j \in \R^{T_j}$,
if $V=(V^1,...,V^k) \in \R^T$, then $\cA^{N,T}_v(V)
\subset \bigcap_{j=1}^k \cA^{N,T_j}_v(V^j)$.

Our main goal is to prove the following:

\begin{theorem} \label {generic}

For every $N \geq 1$, there exists $T \geq 1$ such that for a dense set of
$V \in \R^T$, $\cA^{N,T}_v(V)=\cC^N$.

\end{theorem}

Note that if $\cA^{N,T}_v(V)=\cC^N$, then $V$ is $N$-malleable.
Moreover the property of $\cA^{N,T}_v(V)=\cC^N$ is $G_\delta$,\footnote
{Given a compact set $K \subset \R^N$, the set of $V$ such that
$\cA^{N,T}_v(V) \cap K=\emptyset$ is open.}
so Theorem \ref {generic} implies Theorem \ref {malleable}.

The main step is to establish:

\begin{lemma} \label {existence}

For every $N \geq 1$, there exists $T' \geq 1$ such that for every $E \in \R^N \setminus \cC^N$,
$\cA^{N,T'}_e(E) \neq \R^N$.
%\marginpar{Should the quantifier on $N$ be made explicit? ``For every $N \in \N$, there exists $T' \geq 1$ such that for every $E \in \R^N \setminus \cC^N$,
%$\cA^{N,T'}_e(E) \neq \R^N$.''}

\end{lemma}

\subsection{Proof of Lemma \ref {existence}}

\def\id{\mathrm{id}}
\def\cE{\mathcal{E}}

For $\alpha \in \R$, let $L(\alpha):\sl(2,\R) \to \sl(2,\R)$ be the linear
operator given by $L(\alpha) \cdot z=\begin{pmatrix} \alpha&-1\\1&0
\end{pmatrix}^{-1} z \begin{pmatrix} \alpha&-1\\1&0
\end{pmatrix}$.

Let $c=\begin{pmatrix} 0&0\\1&0 \end{pmatrix}$. We note that $(L(\alpha)^3-L(\alpha)^2) \cdot c=(\alpha^2-2)
(L(\alpha)^2-L(\alpha)) \cdot c-(L(\alpha)-\id) \cdot c$. Let $F(\alpha)$ be the space generated by $(L(\alpha)-\id) \cdot c$ and $(L(\alpha)^2-L(\alpha)) \cdot c$.  If $\alpha \neq 0$, then $F(\alpha)$ is a two dimensional subspace invariant under $L(\alpha)$.

Let now $\cE \subset \R$ be a finite non-empty subset and let $v \in \R$.  Let $L^\cE(v):\sl(2,\R)^\cE \to \sl(2,\R)^\cE$ be given by $L^\cE_e(v) \cdot z=L(e-v) \cdot z_e$, $e \in \cE$.

\begin{lemma} \label {constant}

Let $\cE \subset \R$ be a finite non-empty subset.  Let $v \in \R$ be such
that $2 v \notin \cE+\cE$.
Let $F^\cE(v) \subset \sl(2,\R)^\cE$ be the vector space generated by
$(L^\cE(v)^{j+1}-L^\cE(v)^j) \cdot c^\cE$, $0 \leq j \leq 2 \#\cE-1$.  Then
$F^\cE(v)=\prod_{e \in \cE} F(e-v)$.%\marginpar{Change $F(e-V)$ to $F(e-v)$?}

\end{lemma}

\begin{proof}

The conditions imply that if $e,e' \in \cE$ are distinct,
$\lambda$ is an eigenvalue of $L(e-v)|F(e-v)$,
and $\lambda'$ is an eigenvalue of $L(e'-v)|F(e'-v)$, then $\lambda \neq
\lambda'$.  Moreover, $(L(e-v)-\id) \cdot c$ is cyclic for $L(e-v)|F(e-v)$. 
It follows that $(L^\cE(v)-\id) \cdot c^\cE$ is cyclic for
$L^\cE(v)|\prod_{e \in \cE} F(e-v)$.
\end{proof}

\begin{lemma} \label {two values}

Let $\cE \subset \R$ be a finite non-empty subset.  Let $v,v' \in \R$ be
such that $2 v,2 v' \notin \cE+\cE$ and $v \neq v'$.  Then
$F^\cE(v)+L^\cE(v)^{2 \#\cE+1}(v)) \cdot F^\cE(v')=\sl(2,\R)^\cE$.

\end{lemma}

\begin{proof}

Since $F^\cE(v)$ is $L^\cE(v)$ invariant,
it is enough to show that $F^\cE(v)+F^\cE(v')=\sl(2,\R)^\cE$.  By Lemma \ref
{constant}, it is enough to show that $F(e-v)+F(e-v')=\sl(2,\R)$,
$e \in \cE$, which follows from
$v \neq v'$, $e \neq v$ and $e \neq v'$ by direct computation.
\end{proof}

We can now finish the proof of Lemma \ref {existence}.  Let $E \in \R^N$ and
let $\cE=\{E_j\}_{j=1}^N$.  Let $v,v' \notin \cE+\cE$ be
distinct.  Let $T'=4N+2$, and
let $V \in \R^{T'}$ be given by $V_i=v$, $1 \leq i \leq 2N+1$ and $V_i=v'$,
$2N+2 \leq i \leq 4N+2$.

Let $\Psi:\R^{4N+2} \to \sl(2,\R)^N$ be the linear operator given by
\begin{equation}
\Psi_i=\Phi^{N,4N+2}_i(E,V)^{-1} \partial_V \Phi^{N,4N+2}_i(E,V).
\end{equation}
Then
\begin{equation}
\Psi_i \cdot w=\sum_{j=1}^{2 N+1} w_j L(E_i-v)^{j-1} \cdot
c+w_{j+2N+1} L(E_i-v)^{2N+1} L(E_i-v')^{j-1} \cdot c,
\end{equation}
$1 \leq i \leq N$.  The result
follows from Lemma \ref {two values}.
\qed

\subsection{Proof of Theorem \ref {generic}}

We will need some elementary algebraic geometry; compare,
for example, \cite{H03}. The following is well known (it is the Noetherian
property of polynomial rings over a field, and it is the
result that guarantees that the Zariski topology is a topology):

\begin{lemma}[Algebraic Geometry Lemma]

Let $n \geq 1$ and $p_\lambda$, $\lambda \in \Lambda$, be a non-empty family of polynomials $\C^n \to \C$ and let $\cW=\bigcap_{\lambda \in \Lambda} p_\lambda^{-1}(0)$.  Then there exists a finite $\Lambda_* \subset \Lambda$ such that $\cW=\bigcap_{\lambda \in \Lambda_*} p_\lambda^{-1}(0)$.

\end{lemma}

We will need a refinement for the case of bounded degree:

\def\span{\mathrm {span}}

\begin{lemma}[Quantitative Algebraic Geometry Lemma]

Let $n,d \geq 1$ and let $p_\lambda$, $\lambda \in \Lambda$, be a non-empty family of polynomials $\C^n \to \C$ with degree bounded by $d$, and let $\cW=\bigcap_{\lambda \in \Lambda} p_\lambda^{-1}(0)$.  Then there exists a finite $\Lambda_* \subset \Lambda$ such that $\cW=\bigcap_{\lambda \in \Lambda_*} p_\lambda^{-1}(0)$.  Moreover $\# \Lambda_* \leq k(n,d)$.

\end{lemma}

\begin{proof}

We thank Andrei Mandelshtam for the following argument; compare \cite{M25}.  Let $K$ be a field.
For $Z \subset K[x_1,...,x_n]$, let $\span(Z)$ be the vector space generated
by $Z$ and let $I(Z)$ be the ideal generated by $Z$.
Then $I(Z)=I(\span(Z))$.  In particular $I(Z)=I(Z_*)$ for any $Z_* \subset Z$
with $\span(Z)=\span(Z_*)$, so we can select $Z_*$ with cardinality
$\dim(\span(Z))$.  If $Z$ consists of polynomials of degree at most
$d$, $\span(Z)$ has dimension bounded by $\frac {(n+d)!} {n! d!}$.
\end{proof}

Let $T'$ be as in Lemma \ref {existence}. Note that the sets
$\cA^{N,T'}_v(V)$ are given as the solution of a polynomial equation of
degree bounded in terms of $N$ and $T'$.  By the Quantitative Algebraic
Geometry Lemma, there exists $k$ such that for any non-empty family
$V_\lambda \in \R^{T'}$, $\lambda \in \Lambda$, there exists a
subfamily $\Lambda_*$ with $\#\Lambda_* \leq k$ such that
$\bigcap_{\lambda \in \Lambda}
\cA^{N,T'}_v(V_\lambda)=\bigcap_{\lambda \in \Lambda_*}
\cA^{N,T'}_v(V_\lambda)$.
Let now $T=(k+1) T'$ and let $V=(V_1,...,V_{k+1}) \in \R^T$ with
$V_1,...,V_{k+1} \in \R^{T'}$.  For every $\epsilon>0$,
let $D_j$ be the $\epsilon$-ball around $V_j$.  Let us show
that there exists $V'_j \in D_j$
such that $V'=(V'_1,...,V'_{k+1})$ satisfies $\cA^{N,T}_v(V')=\cC^N$.
By the concatenation property, it is enough to show that
$\bigcap_{j=1}^{k+1} \cA^{N,T'}_v(V'_j)=\CC^N$.  Let us define
inductively a sequence $V^l_j \in D_j$, $1 \leq j \leq k+1$, starting at
$l=0$ and until some $L$ such that $\cW^L=\cC^N$ where
$\cW^l= \bigcap_{j=1}^{k+1} \cA^{N,T'}_v(V^l_j)$.
%\marginpar{Replace $\cW^l= \bigcap_{j=1}^{k+1} \cA^{N,T'}_v(V^l_j)=\cC^N$
%with $\cW^l= \bigcap_{j=1}^{k+1} \cA^{N,T'}_v(V^l_j)$.}
Then we can set $V'_j=V^L_j$ to conclude.

First define $V^0_j=V_j$, $1 \leq j \leq k+1$.  Assume that $V^l_j$,
$1 \leq j \leq k+1$ are defined for some $l \geq 0$.  If $\cW^l=\cC^N$
we set $L=l$ and stop.  Otherwise let $E^l \in \cW^l \setminus \cC^N$.
%\marginpar{Replace $E^l \in \cW^j \setminus \cC^N$ with $E^l \in \cW^l
%\setminus \cC^N$.}
By the choice of $k$, there exists $1 \leq j_l \leq k+1$ such that
$\cW^l=\bigcap_{j \neq j_l}
\cA^{N,T'}_v(V^l_j)$.  By Lemma \ref {existence}, the real algebraic set
$\cA_e^{N,T'}$ is dense in $\R^N$.  Let $V^{l+1}_{j_l} \in D_{j_l}
\setminus \cA^{N,T'}_e(E^l)$
%\marginpar{Replace $D_j \setminus \cA^{N,T'}_e(E^l)$ with $D_{j_l}
%\setminus \cA^{N,T'}_e(E^l)$.}
and let $V^{l+1}_j=V^l_j$ for $j \neq j_l$. Then $\cW^{l+1}
\subset \cW^l \setminus \{E^l\}$.

Since the $\cW^l$ form a strictly decreasing sequence of
algebraic sets, the Algebraic Geometry Lemma implies that
this procedure must terminate in finite time.
\qed

\section{Proof of Theorem \ref{t.3}}\label{sec.4}

In this section we prove Theorem \ref{t.3}. Since for any $V \in \ell^\infty(\Z)$, $\sigma_\mathrm{ess}(H_V)$ is compact, it needs to be shown that for generic $V$, $\sigma_\mathrm{ess}(H_V)$ has empty interior and contains no isolated points. As pointed out above, the first property is a consequence of Theorem~\ref{t.2}. Thus, it remains to show that, generically, $\sigma_\mathrm{ess}(H_V)$ cannot have any isolated points.

For $V \in \ell^\infty(\Z)$, define
$$
s(V) := \sup_{E \in \sigma_\mathrm{ess}(H_V)} \inf_{E' \in \sigma_\mathrm{ess}(H_V) \setminus \{E\}} |E-E'|.
$$
Clearly, $s(V) \ge 0$ and we have $s(V) = 0$ if and only if  $\sigma_\mathrm{ess}(H_V)$ has no isolated points.  Thus our goal is to show that $\{ V \in \ell^\infty(\Z) : s(V) = 0 \}$ is a dense $G_\delta$ set.

It suffices to show that for every $\delta > 0$, $S_\delta := \{ V \in \ell^\infty(\Z) : s(V) < \delta \}$ is open and dense.

We first show that $S_\delta$ is open. It follows from the characterization of the essential spectrum via Weyl sequences that the essential spectrum is $1$-Lipschitz continuous in the Hausdorff metric with respect to $\ell^\infty(\Z)$ perturbations of the potential. Thus, given $V \in \ell^\infty(\Z)$ with $s(V) < \delta$, we claim that for $\varepsilon \in (0, \frac12(\delta - s(V) - \gamma))$, with $\gamma > 0$ small enough that the interval is non-trivial, we have $s(V+W) < \delta$ for every $W$ with $\|W\|_\infty < \varepsilon$. Suppose this fails and there is a $W$ with $\|W\|_\infty < \varepsilon$ and $s(V+W) \ge \delta$. Then there is $E \in \sigma_\mathrm{ess}(H_{V+W})$ such that $\inf_{E' \in \sigma_\mathrm{ess}(H_{V+W}) \setminus \{E\}} |E-E'| \ge s(V+W) - \gamma \ge \delta - \gamma$. But then there is $\tilde E \in \sigma_\mathrm{ess}(H_{V})$ such that 
$$
\inf_{E' \in \sigma_\mathrm{ess}(H_{V}) \setminus \{\tilde E\}} |\tilde E-E'| > \delta - \gamma - (\delta - s(V) - \gamma) = s(V),
$$
contradiction.

Now we show that $S_\delta$ is dense. Let $V \in \ell^\infty(\Z)$ with $s(V) \ge \delta$ (otherwise there is nothing to do) and $\varepsilon > 0$ be given. Since the essential spectrum is bounded, there can be only finitely many isolated points of $\sigma_\mathrm{ess}(H_V)$ with separation at least $\delta$, denote them by $E_1, \ldots, E_\ell$. For all other isolated points of $\sigma_\mathrm{ess}(H_V)$, the separation will be uniformly bounded away from $\delta$, say by $2 \varepsilon > 0$ (otherwise make $\varepsilon$ smaller to accomplish this). Let us construct a perturbation $W$ with $\|W\|_\infty < \varepsilon$, that lowers the separation of $E_1, \ldots, E_\ell$ to a value strictly less than $\delta$. Since the others cannot be lifted to $\delta$ by such a perturbation, it will then follow that $s(V+W) < \delta$. For each of the $E_j$ choose finitely supported Weyl vectors $\psi_{j,k}$ so that their supports are mutually disjoint. For $k$ even, add a small constant value of $W$ (less than $\min \{ \varepsilon, \delta\}$) on the support of $\psi_{j,k}$ with the sign chosen so that the perturbed energy enters the former gap, and for $k$ odd leave $V$ unchanged on that set. This makes the separation of each $E_j$ smaller than $\delta$.\qed

\section{Comments on Effective Perturbations}

With slightly more work the proof of Lemma \ref {existence} yields an
explicit construction of some $V \in \R^T$ such that $\cA^{N,T}_v(V)=\cC^N$.
%\marginpar{$\cA^{N,T}_v(V)$}
Let $T=2(N^2+1)(2N+1)$, and let $V \in \R^T$ be given by
$V_j=v_{[(j-1)/(2N+1)]}$, where $v_i \in \R$, $0 \leq i \leq 2N^2+1$ are
all distinct. For any $E \in \R^N \setminus \cC^N$, in order to check
that $E \notin \cA^{N,T}_v(V)$, we can split $V$ into $N^2+1$ blocks
of length $4N+2$ and for each $E \in \R^N$ at least one of them will
work for the argument, so we can conclude using the concatenation property.

Of course $\cA^{N,T}_e(E)=\R^N$ for every $E \in \R^N$ if $T<3N$.
We have reasonable computer evidence that one can take $T'=3N$ in Lemma \ref
{existence} (our proof gives only $T'=4N+2$).

Indeed $\cA^{N,3N}$ is given
in terms of a single simple equation.
For $(E,V) \in \R^N \times \R^{3N}$, consider the $3N \times 3N$ matrix
$M^N(E,V)=(m^N_{i,j}(E,V))_{1 \leq i,j \leq 3 N}$ such that for
$1 \leq i \leq N$ and $1 \leq j \leq 3N$, $m^N_{3i-2,j}(E,V)$,
$m^N_{3i-1,j}(E,V)$ and $m^N_{3i,j}(E,V)$ are the
upper left, upper right, and lower left coefficients of
$(\Phi^{N,3N}_i)^{-1}(E,V) \partial_j \Phi^{N,3N}_i(E,V)$, where
$\partial_j$ denotes the derivative with respect to the $j$-th coordinate
of $V$. Then $\det M^N(E,V)$ is a polynomial in $E$ and $V$ and its zero
set is $\cA^{N,3N}$.

We have reasonable ($1 \leq N \leq 11$) computer evidence of the following.
For $V=(1,\ldots,3N)$, $\det M^N(E,V+t)$ is a polynomial of degree
$2N^2-N$ in $t$ with coefficients that are polynomials in $E$, and
the coefficient of $t^{2N^2-N}$ is $\prod_{1 \leq i<j \leq N} 4 i
(E_i-E_j)^3$.  If this could be established in general, we could
conclude Lemma \ref {existence} with the optimal $T'=3N$.

In a previous version of this work, we established Lemma \ref {existence} by
leveraging the case $N=2$ (established by the computer calculation) to prove
by induction an abstract result about periodic Schr\"odinger operators:
for every $E \in \R^N \setminus \cC^N$ there exists a periodic Schr\"odinger
operator such that $f:(\P\R^2)^N \to (\P\R^2)^N$ given by
$(x_i)_{i=1}^N \mapsto (f_i(x_i))_{i=1}^N$ is a minimal dynamical system,
where $f_i:\P\R^2 \to \P\R^2$ the projective action of
the monodromy matrix with energy $E_i$.  This result was then used to get
$\cA^{N,T}_e(E) \neq \R^N$ for some $T$ depending on $E$, from which
Lemma \ref {existence} can be concluded using the Algebraic Geometry Lemma.

The proof given in this version is an adaptation of our proof of the
analogue of Lemma \ref {existence} for periodic Schr\"odinger
operators on $\R$, which will appear in our forthcoming work \cite {AD}.

\end{document}